\documentclass{amsart}

\usepackage[left=1in,right=1in,top=1in,bottom=1in]{geometry}
\usepackage[english]{babel}
\usepackage{comment}

\def\phi{\varphi}
\def\tn{\textnormal}

\usepackage{amssymb, mathrsfs, enumerate, amsfonts, latexsym, bbm,
graphicx,amsmath,xspace,amsthm}

\newtheorem{thm}{Theorem}{\bf }{\it }
\newtheorem{statement}[thm]{Statement}{\bf }{\it }
\newtheorem{mythm}[thm]{Theorem}{\bf }{\it }
\newtheorem{prop}[thm]{Proposition}{\bf }{\it }
{\bf }{\it }
\newtheorem{cor}[thm]{Corollary}{\bf }{\it }
\newtheorem{lem}[thm]{Lemma}{\bf }{\it }
\newtheorem{defn}[thm]{Definition}{\bf }{\rm }
\newtheorem{mydef}[thm]{Definition}{\bf }{\rm }
{\bf }{\rm }
\newtheorem{exmp}[thm]{Example}{\bf }{\rm }
{\bf }{\it }
{\bf }{\rm }
{\it }{\rm }

\DeclareMathSymbol{\Gamma}{\mathalpha}{operators}{0}
\DeclareMathSymbol{\Delta}{\mathalpha}{operators}{1}
\DeclareMathSymbol{\Theta}{\mathalpha}{operators}{2}
\DeclareMathSymbol{\Lambda}{\mathalpha}{operators}{3}
\DeclareMathSymbol{\Xi}{\mathalpha}{operators}{4}
\DeclareMathSymbol{\Pi}{\mathalpha}{operators}{5}
\DeclareMathSymbol{\Sigma}{\mathalpha}{operators}{6}
\DeclareMathSymbol{\Upsilon}{\mathalpha}{operators}{7}
\DeclareMathSymbol{\Phi}{\mathalpha}{operators}{8}
\DeclareMathSymbol{\Psi}{\mathalpha}{operators}{9}
\DeclareMathSymbol{\Omega}{\mathalpha}{operators}{10}

\usepackage[dvipsnames]{xcolor} 
\usepackage{tabularx}
\usepackage{soul}

\def\rupert#1{\sethlcolor{LimeGreen}\hl{\textsc{Rupert:} #1}}

\def\da{{\downarrow}}
\def\ua{{\uparrow}}
\def\rt{\mathrm{rt}}

\newcommand{\restrict}{\upharpoonright}

\newcommand\Tot{\mathrm{Tot}}

\newcommand{\G}{{\mathcal{G}}}

\newcommand{\F}{{\mathcal{F}}}

\newcommand{\HIGH}{\ensuremath{\mathsf{HIGH}}\xspace}
\newcommand{\DOM}{\ensuremath{\mathsf{DOM}}\xspace}
\newcommand{\AVOID}{\ensuremath{\mathsf{AVOID}}\xspace}

\newcommand{\HIM}{\ensuremath{\mathsf{HI}}\xspace}

\newcommand{\MAD}{\ensuremath{\mathsf{MAD}}\xspace}

\newcommand{\MED}{\ensuremath{\mathsf{MED}}\xspace}

\newcommand{\BIM}{\ensuremath{\mathsf{BI}}\xspace}

\newcommand{\DNR}{\ensuremath{\mathsf{DNR}}\xspace}
\newcommand{\RCA}{\ensuremath{\mathsf{RCA}_0}\xspace}
\newcommand{\WKL}{\ensuremath{\mathsf{WKL}_0}\xspace}
\newcommand{\ACA}{\ensuremath{\mathsf{ACA}_0}\xspace}

\newcommand{\PA}{\ensuremath{\mathsf{PA}}\xspace}
\newcommand{\ISt}{\ensuremath{\mathsf{I}\mathrm{\Sigma}_2^0}\xspace}
\newcommand{\ISthr}{\ensuremath{\mathsf{I}\mathrm{\Sigma}_3^0}\xspace}
\newcommand{\BSt}{\ensuremath{\mathsf{B}\mathrm{\Sigma}_2^0}\xspace}

\newcommand{\ISo}{\ensuremath{\mathsf{I}\mathrm{\Sigma}_1^0}\xspace}

\newcommand{\BSon}{\ensuremath{\mathsf{B}\mathrm{\Sigma}_1}\xspace}
\newcommand{\ISn}{\ensuremath{\mathsf{I}\mathrm{\Sigma}_n^0}\xspace}
\newcommand{\BSn}{\ensuremath{\mathsf{B}\mathrm{\Sigma}_n^0}\xspace}

\newcommand{\BSnp}{\ensuremath{\mathsf{B}\mathrm{\Sigma}_{n+1}^0}\xspace}

\begin{document}
	
\sloppy

\title{Independence and Induction in Reverse Mathematics}

\author{David Belanger}
\address{School of Physical and Mathematical Sciences, Nanyang Technological University; Singapore}
\email{david.belanger@ntu.edu.sg}
\urladdr{}
%\thanks{}

\author{Chi Tat Chong}
\address{Department of Mathematics, National University of Singapore; Singapore}
\email{chongct@math.nus.edu.sg}
\urladdr{https://blog.nus.edu.sg/chongct/}
%\thanks{}

\author{Rupert H\"olzl}
\address{Fakult\"at für Informatik, Universit\"at der Bundeswehr M\"unchen; Neubiberg, Germany}
\email{r@hoelzl.fr}
\urladdr{https://hoelzl.fr}
%\thanks{}

\author{Frank Stephan}
\address{Department of Mathematics, National University of Singapore; Singapore}
\email{fstephan@comp.nus.edu.sg}
\urladdr{https://www.comp.nus.edu.sg/~fstephan/}
%\thanks{}

\thanks{The investigators acknowledge the following partial support: F.~Stephan's reserach was supported by Singapore Ministry of Education AcRF Tier 2 grant MOE-000538-00 and AcRF Tier 1 grants A-0008454-00-00 and A-0008494-00-00, the grant A-0008494-00-00 also supported the research visits of R.~H\"olzl to the National University of Singapore and the past employment of D.~Belanger at the National University of Singapore.  C.~T.~Chong's research  was partially supported by NUS grant WBS E-146-00-0001-01.}

%\subjclass[2010]{

\begin{abstract}
We continue the project of the study of reverse mathematics principles  inspired by cardinal invariants. In this article in particular we focus on principles encapsulating the existence of large families of objects that are in some sense mutually independent. More precisely, we study the principle $\MAD$ stating that a maximal family of pairwise almost disjoint sets exists; and the principle $\MED$ expressing the existence of a maximal family of functions that are pairwise eventually different. We investigate characterisations of and relations between these principles and some of their variants. It turns out that induction strength at the levels of $\BSt$ or $\ISt$ is an essential parameter; for instance, over $\BSt$, we show that $\neg\MAD$ is equivalent to the principle $\DOM$ expressing that every weakly represented family of functions is dominated by some other function.
%Prior work studied the reverse mathematics of principles inspired by cardinal invariances; however, this work focussed on $\omega$-models. The present work generalises the approach by also investigating the connections to induction principles $B\Sigma^0_2$ and $I\Sigma^0_2$ and gives further insight into the situation for the existence of maximal almost disjoint families of sets $\MAD$ and maximal eventually different families of functions $\MED$. We provide characterisations of these notions and show also that assuming , 
\end{abstract}

\maketitle

\section{Introduction}

Reverse mathematics is the program of classifying the strength of mathematical statements relative to each other over a base system --- see Simpson~\cite{SS}. These mathematical statements are often referred to as  ``principles,'' as in the Ramsey-style ``combinatorial principles'' that have captured much attention since the mid-1990's. 

H\"olzl, Raghavan, Stephan and Zhang~\cite{HRSZ17} introduced a number of principles modelled after notions of cardinal invariants from set theory.
In this article we continue and extend this study with  particular focus on principles that are in some sense related to the mutual independence of objects. More precisely, we use the notion of weakly represented families of objects introduced in the above article~\cite{HRSZ17} to express the existence of large families of objects and then investigate from a reverse-mathematical point of view the strength of principles that assert the existence of maximal families of objects that are ``independent'' of each other in one sense or another. We begin by recalling some basic facts in reverse mathematics.   

\section{Preliminaries} \label{section-preliminaries}
We are interested in models in the language of second-order arithmetic. These take the form $\mathcal{M}=(M,\mathcal S,+,\cdot,0,1,\in)$ where $M$ is the first-order part and $\mathcal S$ is a subset of the power set of $M$.
Typically we suppress the operations and write simply $\mathcal M = (M,\mathcal S)$. We call the elements of $M$ {\em numbers}, with the idea that they stand in for the natural numbers $\omega$ and obey many of the same axioms.
In this article we resist the convention of writing ``sets'' to mean the elements of $\mathcal S$, because we will often need to discuss sets $A \subseteq M$ which are possibly outside of $\mathcal S$. Call a formula a \emph{$\Sigma^A_n$ formula} if it is $\Sigma_n$ while using the given set $A$ as a second-order parameter (which we also call an \emph{oracle}).
Call a set $B \subseteq M$ a \emph{$\Sigma^A_n$ set} if $B$ is definable by a $\Sigma_n^A$ formula. Formulas with multiple oracles are allowed; call a formula or set $\Sigma^0_n$ if it is $\Sigma^{A_1,\ldots,A_k}_n$ for some $A_1,\ldots,A_k \in \mathcal S$; this includes the case where $k=0$ and the list is empty. Analogous definitions with ``$\Pi$'' or ``$\Delta$'' in place of ``$\Sigma$'' can be given in the obvious ways.

Here are some of the axioms we will consider. Let $P^-$ be the axioms of Peano arithmetic not including any induction schemes. 
 Let \ISn denote the induction scheme for $\Sigma^0_n$ formulas, and let \BSn denote the $\Sigma^0_n$-bounding scheme. 
Over $P^-$, Kirby and Paris \cite{KP} proved  the following strict implications  for every $n\ge 1$: 
\[\BSnp\rightarrow \ISn \quad \text{and} \quad\ISn \rightarrow \BSn\] 
(In their account there were no second-order parameters, but the proof is the same.)
 The system $\RCA$ consists of $P^-$ together with \ISo and the comprehension scheme for $\Delta^0_1$ sets. In particular, if $\mathcal M = (M, \mathcal S)$ is a model of $\RCA$, then $S$ is closed under Turing reducibility.

We say a nonempty set $A \subseteq M$ is \emph{$\mathcal M$-finite} if $A$ is in $\mathcal S$ and $A$ has a maximum element. (Our axioms are strong enough to support a number of equivalent definitions of $\mathcal M$-finite, but this one suits our purposes well enough.) We say $A \subseteq M$ is \emph{$\mathcal M$-infinite} if there is an injection $f \in \mathcal S$ from $M$ into $A$. If $\mathcal M$ is a model of $\RCA$, then every $A \in \mathcal S$ is either $\mathcal M$-finite or $\mathcal M$-infinite; but for general $A \subseteq M$, it is possible to be neither $\mathcal M$-finite nor $\mathcal M$-infinite. We say that $A \subseteq \mathcal M$ is \emph{regular} if for each $b \in M$, its initial segment $A \restrict b = A \cap \{0,\ldots,b-1\}$ is $\mathcal M$-finite. If $\mathcal M \models \RCA$ then every element $A \in \mathcal S$ is regular, and in particular every $\mathcal M$-finite set is regular. Using the standard quadratic pairing function $\langle \cdot,\cdot\rangle$ to represent an ordered pair (and again, leaning on the strength of our axiom systems) we can define an appropriate notion of $\mathcal M$-finite strings, that is, sequences of length $m \in M$ and with elements taken from $M$, via $\sigma =  \{\langle x,y\rangle \colon  x <m$ and $\sigma(x)=y\}$, if this set is $\mathcal M$-finite. We also render this sequence as $\sigma=\langle \sigma(0), \ldots, \sigma(m-1)\rangle$, abusing notation in the case $m=2$. We use $M^{<M}$ to denote the set of all $\mathcal M$-finite strings, and $2^{<M}$ to denote the set of all $\mathcal M$-finite strings which consist only of 0's and 1's, i.e.~binary strings.

In this paper, we consider principles related to set theory's cardinal invariants, in the context of  $\RCA$. In particular, we are interested in the role played by the inductive strength of a model of \RCA  in relation to these principles. This naturally requires an analysis of nonstandard models.

The families of functions and families of sets that we consider are the same as those considered by H\"olzl, Raghavan, Stephan and Zhang~\cite{HRSZ17}, and informally speaking consist of those functions that can be written as the ``rows'' $\Psi_e$ of some set~$\Psi$, with the following details: we require that $\Psi$~be $\Sigma^0_1$ (with parameters), as opposed to being an element of~$\mathcal S$; we require that each row~$\Psi_e$ be well-defined as a function, but not necessarily total; and all rows~$\Psi_e$ which are not total are simply ignored and do not contribute to the family. More formally, we will work with the following definition.
\begin{mydef}[Weakly represented families] \label{dafFam}
	Let $\Psi \subseteq M^2 \rightarrow M$ be a $\Sigma^0_1$ (with set parameters) partial function, and for each $e$ define $\Psi_e = \{\langle x,y\rangle \colon  \Psi(e,x) = y\}$. View $\Psi_e$ as a partial function $M \rightarrow M$. Define $\F =\{\Psi_e \colon  \Psi_e~\tn{is~total}\}$. Then $\F$ is a family of total functions $M \rightarrow M$; we say that $\F$ is \emph{weakly represented by $\Psi$}. In this case we call $\F$ a \emph{weakly represented family}.
\end{mydef}
%	An alternative name could be to call $e \mapsto \Psi_e$ a `numbering' of $\F$ which ignores all $\Psi_e$ which are not total.
 By identifying a set $A$ with its characteristic function, we can naturally speak of a family~$\F$ of sets~$A \subseteq M$ being weakly represented. We mention that these notions of weakly represented families also were used by H\"olzl, Jain and Stephan~\cite{HJS} to study inductive inference in the setting of reverse mathematics.
 
If $\F$ is a weakly represented family (of functions or of sets), then each element $f \in \F$ is $\Delta^0_1$ relative to sets in $S$  and therefore is in $\mathcal S$. By considering the model~$\mathcal M$'s version of the universal oracle Turing machine, we obtain the following result.
\begin{lem} \label{lemma-universal}
	For each $A \in \mathcal S$ there exists a $\Sigma^A_1$ partial function $\Phi^A \colon  M^2 \rightarrow M$ which is \emph{universal} in the sense that for every $\Sigma^A_1$ partial function $\Psi^A \colon  M^2 \rightarrow M$ there is a $\Delta_1^A$ function $f$ satisfying $\Psi^A_e = \Phi^A_{f(e)}$.
\end{lem}
	We use $\F^A_{\Tot}$ to denote the family of functions weakly represented by $\Phi^A$. Then every weakly represented family $\F$ is contained in $\F^A_{\Tot}$ for some $A \in \mathcal S$.
% and whenever $\F = \{\Phi_e^A : \Phi_e^A$ is total and $e\in E\}$, we call $E$ an \emph{indexing} of $\F$ in $\Phi^A$. 
Lastly, we fix a definition of \emph{finite} that is suitable for weakly represented families.
\begin{defn}\label{infinite-family}
 A weakly represented family $\F$ is called \emph{finite} if there is a $\Psi$ which weakly represents it, and an $\mathcal M$-finite set $E$ such that $\F = \{\Psi_e \colon e \in E\}$. Otherwise, $\F$ is called \emph{non-finite}.
\end{defn}
In the absence of stronger induction axioms, a family $\F$ which is \emph{non-finite} in this sense does not satisfy certain intuitive properties of infinity. As an example: If $\neg \ISt$ holds, then the family $\F$ conisting of functions $f_e(x) = e$ if $\Phi_e^A(x) \da$ for $e < e_0$ is bounded, but in general not finite; this is because $\{e < e_0\colon \Phi_e^A \tn{is~total}\}$ is in general not $M$-finite. We bring this up to explain the presence of the unusual term {\em non-finite}, as well as to point out that certain of our results below depend on this definition, and could end up different if some other notion of {\em finite} were used.

The rest of this paper is structured as follows.
Section \ref{section-dom} deals with the principle $\DOM$ which expresses essentially that there is a function dominating every total  function in a weakly represented family; we prove a version of Martin's theorem \cite{Martin66a} characterizing such functions in terms of high Turing degrees.
%, and obtain a surprising result comparing $\DOM$ with levels $\ISn$ of arithmetical induction. 
Section \ref{section-key} is a technical section where we carry out a gener        al $0''$ tree argument for use in later sections, to construct new weakly represented families. Section \ref{section-mad} deals with the principle $\MAD$, which says that there is a maximal almost disjoint (MAD) family of sets. We obtain results comparing $\MAD$ with $\DOM$, and with $\ISt$. In Section~\ref{section-med}, we investigate the principles $\MED$, which says there is a maximal eventually different (MED) family of functions, as well as $\AVOID$, which says essentially that for every weakly represented family there is a function eventually different (ED) from all functions in the family; we obtain results about the relationship between the two. Section~\ref{section-ndf} covers  the principle $\HIM$, which says essentially that there is a hyperimmune Turing degree. Finally, Section~\ref{section-mind} is concerned with the principle $\BIM$, which says  that there is a bi-immune Turing degree; we use the existence of a low bi-immune Turing degree to obtain a conservation result.

\section{Dominating families of functions}\label{section-dom}

\begin{statement}[Domination principle, \DOM] Given any weakly
	represented family $\mathcal F$  of 
	functions, there exists a function $g $ such that $g$ \emph{dominates} $\mathcal F$, that is, for every
	$f\in \mathcal{F}$,  we have  $g(x)> f(x)$  for all sufficiently large $x$.
\end{statement}
	 
In an $\omega$-model --- meaning a model $(M,\mathcal S)$ in which $M$ is the true natural numbers --- if $X\in\mathcal S$ computes a function which dominates every total recursive function, then $X$~is generalised high, that is, $X'\equiv_T (X\oplus \emptyset')'$. This was first proven by Martin \cite{Martin66a}.
The phenomenon is reflected in $\RCA$ by the following proposition, proven essentially by Chong, Qian, Slaman and Yang \cite[Lemma 2]{CQSY}: 

\begin{prop}[Martin's Theorem, formalised]\label{martins-theorem}
For a model $(M,\mathcal{S})$, the following are equivalent over $\RCA$:
\begin{enumerate}
   \item[(i)] $\DOM{:}$ For every $\mathcal F$ that is a weakly represented family in $\mathcal{S}$ there is a function $g \in \mathcal{S}$ which grows faster than every $f \in \mathcal F$.
   \item[(ii)] $\HIGH{:}$ For every $A\in \mathcal{S}$ there is a $B\in \mathcal{S}$ such that
     every $\Sigma^A_2$ set $C \subseteq M$ is $\Delta^B_2$.
   \item[(iii)] For each $A\in \mathcal{S}$ and every $\Sigma^A_2$ set $C \subseteq M$
     there is a $B\in \mathcal{S}$ such that $C$ is $\Delta^B_2$.
\end{enumerate}
Note that we do not assume $C \in \mathcal{S}$.
\end{prop}
\begin{proof}
(i $\Rightarrow$ ii) Fix $A $ and let $f $ be a total function which dominates all total $\Delta_1^A$~functions. Take any $C$ which is $\Pi_2^A$, and let $\psi^A$ be the $\Delta_0^A$ formula satisfying
 \[x \in C \iff (\forall y)(\exists z)[\psi^A(x,y,z)].\]
Define $g(x,y)$ to be the partial $\Sigma_1^A$ function mapping each pair $\langle x,y\rangle$ to the least $z$ satisfying~$\psi^A(x,y,z)$. Then for a given $x$, we have $x \in C$ iff the function $\lambda y.g(x,y)$ is total $\Delta_1^A$\!. Hence,
  \begin{equation*}\begin{split}
	  x\in C \iff (\exists y_0) (\exists z_0)\left[
	  \begin{array}{c}
	  	(\forall y > y_0)(\exists z < f(y))[\psi^A(x,y,z)]\\
	  	\tn{and}~(\forall y \leq y_0) (\exists z< z_0)[ \psi^A(x,y,z)]
	  \end{array}
	  \right],
  \end{split}\end{equation*}
 meaning $C$ is $\Sigma_2^f$ and hence   $\Delta_2^f$.

\medskip

\noindent(ii $\Rightarrow$ iii) Immediate; the difference between (ii) and (iii) is that whereas in (ii) a single~$B$ is assigned to each $A$, in (iii) $B$ is allowed to vary with each $\Sigma^A_2$-definable set.

\medskip

		 \noindent(iii $\Rightarrow$ i) Suppose $\F$ is a weakly represented family of functions. Then $\F \subseteq \F^A_{\Tot}$ for some~$A$; so it suffices to exhibit some $g$ which dominates all $f \in F^A_{\Tot} = \{\Phi_e^A \colon  \Phi_e^A$ is total$\}$. 
		 
		 Since the predicate ``$\Phi_e^A$ is total'' is $\Pi^A_2$, there exists by (iii) a $B $ and a $\Sigma^B_0$ formula $\psi$ satisfying
	 \[\Phi^A_e\tn{~is~total} \iff (\exists s)(\forall t) [\psi^B(e,s,t)].\]
  Then  for every input $x$ let $g(x)$ be the least~$t_0>x$ such that 
  \[(\forall e < x)\left[\tn{either}~\Phi_{e,t_0}^A(x) \da ~\tn{or}~(\forall s < x)(\exists t < t_0) [\neg \psi^A(e,s,t)]\right].\]
  Then $g$ is well-defined, $\Sigma_1^B$ and total.   By definition it dominates each total $\Delta_1^A$~function.
\end{proof}
This formalised version of Martin's Theorem allows for a short proof of the previously known fact that  $\DOM$ and $\BSt$ together imply full arithmetical induction and thus Peano Arithmetic $(\PA)$.
\begin{mythm}[H\"olzl, Raghavan, Stephan, Zhang~\cite{HRSZ17}] \label{dom-induction}
	$\RCA+\DOM+\BSt$ implies $\ISn$ for every $n\ge 1$. In short,
$\RCA+\DOM+\BSt \vdash \PA$.
\end{mythm}
We will use the following well-known fact in the new proof.
\goodbreak
 \begin{lem} \label{lem:regular}
  The following statements hold for every model $(M, \mathcal S)$ of $\RCA$:
  \begin{enumerate}
   \item $\ISn$ iff every $\Sigma^0_n$ set is regular iff every $A^{(n)}$ is regular for  $A \in \mathcal S$,
   \item $\BSn$ iff every $\Delta^0_n$ set is regular iff every $B \leq_T A^{(n-1)}$ is regular for $A \in \mathcal S$.
  \end{enumerate}
 \end{lem}
%\noindent
% It turns out that $\DOM$, together with $\BSt$, implies full arithmetical induction.
%\begin{mythm} \label{dom-induction}
% $\RCA+\DOM+\BSt$ implies $\ISn$ for each $n\ge 1$.
%\end{mythm}
\begin{proof}[Proof of Theorem~\ref{dom-induction}]
Fix $(M,\mathcal S)\models \text{RCA}_0+\DOM+\BSt$ and let $A_0 \in\mathcal S$.  Use Proposition~\ref{martins-theorem} to obtain a sequence $\{A_n\colon n\in\omega \}$ such that $A_{n+1}$ is 
 high relative to $A_n$. In particular, the double jump $A_n''$ is $\Delta_2^{A_{n+1}}$ and therefore regular by $\BSt$ and Lemma~\ref{lem:regular}.
	Since $A_0'' \leq_T A_1'$, we also have $A_0''' \leq_T A_1''$. Hence by transitivity ($\BSon$ relative to $A_2'$), we have $A_0''' \leq_T A_2'$ and hence that $A_0'''$ is regular. Iterating this argument for larger and larger $n$, we find that $A_0^{(n)} \leq_T A_{n-1}'$ and $A_0^{(n)}$ is regular for all $n\in \omega$. The result follows from Lemma~\ref{lem:regular} and the fact that $A_0$ was arbitrary.
\end{proof}
 Thus it is all the more surprising that, by the following theorem, $\DOM$ does not imply~$\BSt$. The following theorem also gives an alternative (but similar) proof that $\HIGH$ does not imply $\ACA$, by way of conservation in place of cone avoidance. Recall that a system  $T_1$~is $\Pi^1_1$\nobreakdash-conservative over a system $T_2$  if $T_1\supset T_2$ and every $\Pi^1_1$-sentence provable in $T_1$ is provable in~$T_2$. By a well-known observation of Harrington, one can establish $\Pi^1_1$-conservation by showing that every countable model $\mathcal M=(M,\mathcal S)$ of $T_2$ may be extended to a model $\mathcal M'=(M,\mathcal S')$ of $T_1$ having the same first-order part $M$ and such that $\mathcal S\subset\mathcal S'$.

  \begin{mythm}\label{DOM-1} $\RCA+\DOM$ is $\Pi^1_1$-conservative over $\RCA$.
  \end{mythm}
  
  \begin{proof}
	  Let $\mathcal M_0 = (M,\mathcal S_0)$ be a countable model of $\RCA$ topped by a function $f_0 \in\mathcal S$; that is, every $A \in \mathcal S$ is Turing-reducible to $f_0$. We will construct a sequence of functions $(f_n)_{n < \omega}$, and with it a sequence of countable models $( \mathcal M_n )_{n < \omega}$ given by 
	  \[\mathcal M_{n+1} = \mathcal M_n[f_{n+1}] = (M,\mathcal S_{n+1}),\]
	  where $\mathcal S_{n+1}$ is the closure of $\mathcal S \cup \{f_{n+1}\}$ under $\Delta^0_1$ comprehension.
 We ensure that each $f_{n+1}$ dominates $\F_{\Tot}^{f_n}$ and each $\mathcal M_n$ satisfies $\ISo$ (and, in particular, $\mathcal M_n$ is topped by $f_n$). Then the limiting model $\mathcal M_\omega = (M, \bigcup_{n < \omega} \mathcal S_n)$ will satisfy $\RCA + \DOM$, proving the theorem by Harrington's method.

	 For the inductive step, begin with the countable model $\mathcal M_n \models \ISo$ topped by $f_n$, and consider the class $\F_{\Tot}^{f_n}$. Since $\mathcal M_n$ is countable, we may (externally to the model) let $(g_i)_{i < \omega}$ be a list of all $\Delta_1^{f_n}$ total functions; let $(\phi_i^{f_n})_{i < \omega}$ be a list of all $\Sigma_1^{f_n}$ formulae; and let $(c_k)_{k < \omega}$ be an increasing sequence that is cofinal in $M$. We construct $f_{n+1}$ by initial segments $\sigma_k \in M^{<M}$\!\!. Beginning with $\sigma_0$ as the empty string, for each $k \in \omega$ we choose the next $\sigma_{k+1}$ to satisfy
\begin{itemize}
	\item[(i)] $\sigma_{k+1}$ is a proper extension of $\sigma_k$,
	\item[(ii)] $|\sigma_{k+1}| > c_{k}$,
	\item[(iii)] $\sigma_{k+1}(x) \geq g_i(x)$ whenever $i < k$ and $|\sigma_k| < x \leq |\sigma_{k+1}|$, and
	\item[(iv)] either $\mathcal M_n \models \phi_k^{f_n}(\sigma_{k+1})$ or $\mathcal M_n \models (\forall \tau \supseteq \sigma_{k+1}) [\neg \phi_k^{f_n}(\tau)]$.
\end{itemize}
  Because $k$ is always in $\omega$, some suitable $\sigma_{k+1}$ exists at every stage. This completes the construction.

	  The requirements (i) and (ii) together ensure that $f$ is a well-defined, total function (rather than having a proper cut as its domain). Item (iii) ensures that $f_{n+1}$ dominates each function in $\mathcal M_n$, and (iv) is a form of ``forcing the jump'' that ensures that $\mathcal M_n[f_{n+1}]$ models~\ISo.
	  This completes the proof.
  \end{proof}

 \section{Handling index sets}\label{section-key}
 This technical section lists some methods for producing new weakly represented families $\G$ from known families $\F$, in models of $\ISt$. 
%\begin{mydef}\label{def:indexing}
%	Suppose $\F$ is a family weakly represented by $\Psi$, and suppose we have $\G \subseteq \F$. If $B \subseteq M$ is such that $\G = \{\Psi_e : e \in B $ and $\Psi_e$ is total$\}$, then we call $B$ an \emph{indexing} of $\G$ in $\Psi$.
%\end{mydef}
% The difference between this \emph{indexing} and the \emph{indexing} introduced after Lemma \ref{lemma-universal} is that here a general $\F$ takes the place of the more specific, universal $\F^A_{\Tot}$. In this section we present some sufficient conditions on the $B$ of Definition \ref{def:indexing} which guarantee that $\G$ is weakly representable.
These take the form of sufficient conditions on the set $B$ in the definition
 $\G = \{\Psi_e \colon  e \in B\}$,
where $\F$ is weakly represented by $\Psi$. However, the constructions are not as simple as just restricting $\Psi$ to the rows in $B$, because we would like $B$ to be of higher complexity than that would allow. Instead, we use a conventional $0''$~infinite injury argument --- reminiscent of, for instance, Yates~\cite{Y69} --- to get a new $\Theta$ which weakly represents $\G$, and which can be viewed as a scrambled version of $\Psi$ restricted to $B$.

We begin with an easy but helpful lemma; here, we let
\[\rt B = \{\sigma \in 2^{<M} \colon  \sigma(b) = 0 \rightarrow b \in B\},\]
that is, $\rt B$ consists of the binary strings which are either substrings of the complement $\overline B$'s characteristic function, or lexicographically to the right of that characteristic function. 
 \begin{lem}\label{pi2-strings}
  Let $(M,\mathcal S)$ be a model of $\RCA + \ISt$ and assume that $B \subseteq M$ is $\Pi^A_2$ for some~$A$. Then $\rt B$ is~$\Pi^A_2$.
 \end{lem}
 \begin{proof}
  Fix $A$ such that $B$ is $\Pi_2^A$. Then $B$ is $\Pi_1^{A'}$. Then clearly $\rt B$ is $\Pi_1^{A'}$. So $\rt B$ is $\Pi_2^A$.
 \end{proof}
 Next we present three new methods of representing families of functions under different assumptions; in each case, the $\Theta$ we are constructing can be thought of as a restricted and scrambled version of the input representation $\Psi$.

 \begin{lem}[$\Pi^0_2$ weak representation]\label{key-lemma}
	 Suppose $\mathcal M=(M,\mathcal S)$ is a model of $\RCA + \ISt$ and $\Psi$ is $\Sigma^0_1$. If~$B \subseteq M$~is $\Pi^0_2$, then the family 
	 $\{\Psi_e \colon  e \in B \tn{ and } \Psi_e \tn{ is total}\}$
	 is weakly represented.% by a $\Sigma^0_1$ set $\Theta$. Furthermore each $e$ gives rise to at most one total row in $\Theta$.
 \end{lem}
 \begin{proof}
	 Let $\rt B$ be as in Lemma \ref{pi2-strings}. Since $\rt B$ is $\Pi^0_2$, it has a $\Delta^0_1$ approximation $(\sigma_s)_{s \in M}$ satisfying $\sigma \in \rt B$ iff $(\exists^\infty s)[\sigma = \sigma_s]$. Begin the construction of a new $\Sigma^0_1$ set $\Theta$ by assigning to each $\sigma \in 2^{<M}$ a row $\Theta_\sigma$, while leaving $\mathcal M$-infinitely many rows of $\Theta$ unassigned. Then proceed by stages as follows:

\medskip

  At even stages $2s$, for each $\tau$ strictly to the right of $\sigma_s$, abandon its currently assigned row $\Theta_\tau$ and assign to it a new, fresh row (which now takes the name $\Theta_\tau$). Leave $\mathcal M$-infinitely many rows of $\Theta$ still unassigned.

\medskip

  At odd stage $2s+1$, for each $e < |\sigma_s|$ for which $\sigma_s(e) = 0$, make $\Theta_{\sigma_s \upharpoonright e}$ imitate $\Psi_e$ for $s$-many steps of computation. 

\medskip

	 This completes the construction. Clearly $\Theta$ is a partial function, and every total row of~$\Theta$ is the row $\Theta_{\sigma}$ eventually assigned to some $\sigma$ which is an initial segment of the complement $\overline{B}$'s characteristic function. And on the other hand, such a $\Theta_{\sigma}$ is total if and only if $\sigma(e) = 0$ (meaning $|\sigma|-1 \in B$) and $\Psi_{|\sigma|-1}$ is total. Hence $\Theta$ weakly represents the desired family of functions.
 \end{proof}

\begin{lem}[$\Sigma^0_2$ weak representation]\label{key-lemma-2}
	As in Lemma \ref{key-lemma}, but with $B$ being $\Sigma^0_2$ instead of $\Pi^0_2$.
\end{lem}
	This can be seen by almost the same proof as above, but using $\overline B$ as the $\Pi^0_2$ set and accordingly replacing ``$\sigma_s(e)=0$'' by ``$\sigma_s(e)=1$'' in the odd stages.

\begin{lem}[$\Pi^0_2$ dependent weak representation] \label{key-lemma-3}
	Suppose $\mathcal M=(M,\mathcal S)$ models ${\RCA + \ISt}$ and that $\phi$ is a~$\Pi^0_2$~predicate. Then 
	\[\{\Psi_{e_i}\colon  e_i \tn{ is least such that }\phi(\langle e_0,\ldots,e_{i-1}\rangle) \tn{ holds}\}\] is a weakly represented family.
\end{lem}
 \noindent
 In other words: $e_0$ is least such that $\phi(\langle e_0\rangle)$ holds, then $e_1$ is least such that $\phi(\langle e_0,e_1\rangle$) holds, and so on. It is possible that $e_i$ is defined for all $i \in M$; it is possible that there is no $i$ at all, or a maximum $i$ for which $e_i$ is defined; and it is also possible that the set of $i$ forms a proper cut.

 \begin{proof}
  Associate to each $\sigma \in 2^{<M}$ the sequence $\langle e_0^\sigma,\ldots,e_k^\sigma\rangle \in M^{<M}$ where
  \begin{itemize}
  	\item $e_0^\sigma$ is the number of $1$'s occurring before the first $0$ in $\sigma$, and 
   \item $e_{i+1}^\sigma$ is the number of $1$'s occurring between the $(i+1)$-th and the $(i+2)$-th $0$ in $\sigma$; and
   \item for each $\sigma$, we use $k$ to mean the largest number for which $e_k^\sigma$ is defined.
  \end{itemize}
   Then the set of $\sigma$ for which $\mathcal M \models \phi(\langle e_0^\sigma,\ldots,e_k^\sigma\rangle)$ holds is closed under initial segment --- in other words, it forms a binary tree --- it is $\Pi^0_2$ definable, and its leftmost infinite path $\sigma$ corresponds to the sequence $\langle e_0,e_1,\ldots\rangle$ described in the lemma's statement. By applying the argument of Lemma \ref{key-lemma} to this tree and its $e_i$, we obtain the required $\Theta$.
 \end{proof}

\section{Maximal almost disjoint families of sets}\label{section-mad}

Two sets $A, B \subseteq M$ are said to be \emph{almost disjoint} if $A \cap B$ is $\mathcal M$-finite.
A family  $\F$ of subsets of $M$ is  \emph{almost disjoint} if any two distinct elements of $\F$ are almost disjoint.
 $\F$ is called \emph{maximal almost disjoint} if it is non-finite, almost disjoint and  not properly contained in any almost disjoint family.

\begin{statement}[Maximal Almost Disjoint Family principle, \MAD] 
	There is a non-finite, weakly represented family  of infinite sets 
	that is maximal almost disjoint.
\end{statement}

H\"{o}lzl, Raghavan, Stephan and Zhang \cite{HRSZ17} showed that $\MAD$ is equivalent to $\neg \DOM$ in $\omega$-models. We extend this result to  all models of $\RCA+ \ISt$.

\begin{mythm}\label{dom-iff-not-mad}
 $\RCA + \ISt \vdash \DOM \leftrightarrow \neg \MAD$.
\end{mythm}

\begin{proof}
($\DOM \rightarrow \neg \MAD$) Fix a model $\mathcal M = (M, \mathcal S) \models \RCA + \ISt + \DOM$, and any non-finite almost disjoint family $\F$ which is weakly represented, say by $\Psi^A$ for some $A \in \mathcal S$ (note that since  $\ISt$ holds in the model,  $\F$ is in fact $\mathcal M$-infinite). The set
 \[ E=\{e \colon  \Psi^A_e~\mbox{is total and (the characteristic function of) an $M$-infinite set}\}\]
 is $\Pi^A_2$, and hence by $\DOM$ and Proposition \ref{martins-theorem} it is $\Delta^f_2$ for some $f \in \mathcal S$. Let $(E_s)_{s \in M}$ be a $\Delta^f_1$ approximation to~$E$.
%  as in the Limit Lemma.

Define a new set $B = \{b_0 < b_1 < \ldots\}$ as follows. Let $b_0$ be any number. To define $b_{k+1}$, search for a pair~${x,s > b_k}$ satisfying
  \[(\forall e<k)[\tn{either} ~\Psi_{e,s}^A(x){\downarrow} = 0~\tn{or}~e \not \in E_s],\]
	and let $b_{k+1}$ be this $x$. Since $B$ is being enumerated in increasing order and in a $\Sigma^f_1$ way, we know by~$\ISo$ that $B \in \mathcal S$ and (using the definition of $E$) that $b_k$ is defined for every~$k \in M$. And since $(E_s)_{s \in M}$ reaches a pointwise limit, $\mathcal M \models (\forall e)(\forall^\infty k)[b_k \not\in \Psi_e^A]$. Hence, if we let $\Theta$ equal $\Psi$ prepended with a row for $B$'s characteristic function, we see that $\F \cup \{B\}$ is a weakly represented family strictly larger than $\F$. Now fix any $C \in \F$; we claim that $B$ and $C$ are almost disjoint. Suppose for a contradiction that $B \cap C$ were $\mathcal M$-infinite, and let $\Psi^A_e$ be $C$'s characteristic function. Then from the definition of $b_k$ we know that $e \not \in E_s$ for unboundedly many $s$, and so in the limit we have $e \not \in E$. But this contradicts the definition of $E$. Therefore $\F$ is contained in the strictly larger almost disjoint family $\F \cup \{B\}$, and in particular $\F$ is not a MAD family. Since $\F$ was arbitrary, we conclude $\neg \MAD$ holds.

\medskip

\noindent ($\neg \MAD \rightarrow \DOM$)
	Fix any model $\mathcal M = (M,\mathcal S) \models \RCA + \ISt + \neg \MAD$, fix any $A \in\mathcal S$, and let $\Phi^A$ be universal as in Lemma \ref{lemma-universal}. 
 Consider the formula $\psi$ which takes a tuple $\langle e_0,\ldots,e_k\rangle$ of indices as input and expresses that
 \begin{itemize}
  \item[(a)] each $\Phi^A_{e_i}$ is the characteristic function of an $\mathcal M$-infinite set --- call it $C_i$;
  \item[(b)] these $C_i$'s are pairwise disjoint;
  \item[(c)] the union $D_k = \bigcup_{i < k}C_i$ is $\mathcal M$-coinfinite; and
  \item[(d)] for all $x$ and $e<i<k$, either $\Phi^A_e$ is not total or $\Phi^A_e(x)$ is less than the $x$-th smallest element of $ \overline{ D_i}$.
 \end{itemize}
	These statements are $\Pi^0_2$; for (a)--(c) that is straightforward; for (d) consider the formula 
	\[(\forall x)(\forall i < k)(\forall s)(\exists t>s)(\exists y)(\forall z<y)(\forall e< i)[z \in D_{i,t} \wedge (\Phi^A_{e,t}(x)\da<y \vee e \in \Tot^A_t)],\] 
	where $(\Tot^A_t)_t$ is the usual $A$-computable approximation to $\Tot^A= \{e\in M \colon  \Phi_e^A$ is total$\}$.

 Thus, applying Lemma~\ref{key-lemma-3}, we obtain some $\Theta$ weakly representing
  \[\F = \{\Phi^A_{e_i} \colon  e_i \in \langle e_0,e_1,\ldots\rangle\},\]
where $\langle e_0, e_1, \ldots\rangle$ is lexicographically least satisfying both $\psi$ and the relation $e_0 < e_1 < \cdots$.

  Now $\F$ forms a partition of $M$ into $\mathcal{M}$-infinite sets, and hence is an almost disjoint family. Since ${\mathcal M \models \neg \MAD}$, there is a non-finite $B \in\mathcal S$ almost disjoint from each $C_i$. By $\BSt$ this $B$ is almost disjoint from each $D_i$, meaning all but $\mathcal M$-finitely much of $B$ is contained in each complement $\overline{ D_i}$. Hence the function $p_B$ which maps $x$ to the $x$-th element of $B$ dominates all total $\Phi_e^A$ for $e < i$. By allowing~$i$ to vary, we see that $p_B$ dominates all total $\Delta^A_1$ functions. Since $A \in\mathcal S$ was arbitrary, we conclude that $\DOM$~holds.
\end{proof}
 An alternate proof of $\neg \MAD \rightarrow \DOM$ might construct the partition directly using a $0''$ tree argument similar to the proof of Lemma \ref{key-lemma}.
 
\goodbreak

\begin{prop}\label{prop:mad-no-ist}
 $\RCA + \BSt + \neg \ISt \vdash \MAD$.
%In particular, $\RCA+\BSt \vdash \neg \MAD \leftrightarrow \DOM$.
\end{prop}
\begin{proof}
	Fix a model $\mathcal M = (M, \mathcal S)$ and let $A \in \mathcal S$ be such that there is a $\Pi^A_2$ set $B$ which is bounded from above but  not $\mathcal M$-finite; such $A,B$ are guaranteed to exist by $\neg \ISt$ and Lemma~\ref{lem:regular}. It is routine to use the $\Pi^A_2$ formula defining $B$   and its upper bound $a \in M$ to define a partition $ C_0,\ldots,C_{a-1} \in \mathcal{S}$ of $M$ such that $C_i$ is $\mathcal M$-infinite iff $i \in B$. Let $\Psi^A \colon  M^2 \rightarrow M$ be the $\Sigma^A_1$ partial function
	\[\Psi^A(i,s)=\begin{cases}
		1 &\tn{if } i < a\tn{ and } s \in C_i, \\
		0 &\tn{if } i < a\tn{ and } s \not \in C_i\tn{ and }(\exists s>t)[t \in C_i], \\
		\ua &\tn{otherwise.} \\
	\end{cases}\]
Then $\Psi^A$ weakly represents the family $\F = \{C_i \colon  i \in B\}$. Since the $C_i$ form a partition, $\F$~is almost disjoint; we claim that $\F$ is MAD. To see this, suppose for a contradiction that some $D\in \mathcal S$ is $\mathcal M$-infinite, but ${D \cap C_i}$~is $\mathcal M$-finite for every $i \in B$. Then $D \cap C_i$ is $\mathcal M$-finite for every $i < a$. That is, we have 
\[(\forall i < a)(\exists s)(\forall t > s)[t \not \in C_i \cap D].\]
Using $\BSt$ to exchange the two leading quantifiers yields $(\exists s)(\forall i < a)(\forall t > s)[t \not \in C_i \cap D]$. Since the $C_i$ form a partition, this implies that $D$ is $\mathcal M$-finite, a contradiction.

	To complete the argument, we claim that $\F$ is a non-finite family. Assume for a contradiction that it is finite, that is, that there exists some other $\Theta$ weakly representing $\F$ such that there is an $\mathcal M$-finite set~$E$ with $\F = \{\Theta_e \colon  e \in E\}$. Then, for each $e \in E$, we can find one $w_e \in \Theta_e$ and using $A$ as an oracle the unique~$i$ such that $w_e \in C_i$. But this gives an effective procedure for determining which $i<a$ are elements of $B$. Thus $B$ would be $\mathcal M$-finite, contradicting our assumption.
\end{proof}
As a corollary, we obtain an extension of Theorem~\ref{dom-iff-not-mad}.
\begin{cor}
	$\RCA + \BSt \vdash \DOM \leftrightarrow \neg \MAD$.
\end{cor}
\begin{proof}
As $\RCA+\BSt+\DOM \vdash \ISt$, no model of $\RCA+\BSt+\neg\ISt$ satisfies $\DOM$. But by Theorem~\ref{dom-iff-not-mad} each such model satisfies~$\MAD$. 
\end{proof}

\section{Avoidance and eventual difference}\label{section-med}

Two functions are {\em eventually different} if there is an upper bound on the inputs where they are equal; they are {\em eventually equal} if there is an upper bound on the inputs where they are different.

\begin{statement}[Avoidance principle, \AVOID] Given any weakly
	represented family of 
	functions $\mathcal{F}$, there exists a function $g $ that is eventually different from all $f\in \mathcal{F}$.
\end{statement}

\begin{statement}[Maximal Eventually Different Family principle, \MED]
	There is a weakly represented family $\F$ of functions such that
	\begin{itemize}
		\item any two functions in $\F$ are
                  either equal or eventually different, and
		\item every function coincides
		  infinitely often with some function in $\F$.
	\end{itemize}
\end{statement}
	
 \begin{mythm}
  $\RCA \vdash \MED \rightarrow \neg \AVOID$.
 \end{mythm}
 \begin{proof}
Clearly, any family witnessing $\MED$ is a counterexample to $\AVOID$.
\end{proof}

The following proposition is a reformulation of a previously known result
in the terminology of reverse mathematics; we include our own proof
since as far as we know it is more direct than others that have appeared
in print.
% Kjos-Hanssen ca.\ 2003 obtained an interesting characterization of the degree-theoretic analogue of $\AVOID$, later published in joint work in 2006. What follows is that characterization, formulated in reverse-mathematical terms. 
 \begin{prop}[%
% 	Kjos-Hanssen \cite{KH}; 
 	Kjos-Hanssen, Merkle, Stephan \cite{KMS}] $\RCA \vdash (\DOM \vee \DNR) \leftrightarrow \AVOID$.
 \end{prop}
 \begin{proof}
That $\DOM$ implies $\AVOID$ is immediate from the definitions. To see that $\DNR$ implies $\AVOID$, first suppose that $f$ is a DNR function relative to a given $A \in \mathcal S$, that is, that $(\forall e)[\Phi^A_e(e) \neq f(e)]$. Let $\Phi^A$ be universal in the sense of Lemma~\ref{lemma-universal}, and for each $e$ 
 let $p_e$ be the $e$-th prime number, starting with $p_0=2$ and possibly continuing into the nonstandard numbers. 
Define a function $g \leq_T f$ as follows.
Given~$e$, find a sequence~${i_0,\ldots,i_e}$ of indices of partial functions, for each $j$ satisfying
	 \[
	 \Phi^A_{i_j}(i_j) = \begin{cases}
		 0 &\tn{if } \Phi^A_j(e) =0, \\
		 k &\tn{if } k \tn{ is greatest such that } p_j^k \tn{ divides } \Phi^A_j(e), \\
	  \end{cases}
	 \]
and let $g(e) = p_0^{f(i_0)} p_1^{f(i_1)} \cdots p_e^{f(i_e)}$. This ensures that $g(e) \neq\Phi^A_j(e)$ for every~${j \leq e}$; and thus for every~${j \in M}$, we have that $g$ is eventually different from $\Phi^A_j$.
 
 \medskip
 
For the other direction, we need to show that $\AVOID$ implies $\DOM \vee \DNR$. So suppose that $\AVOID + \neg \DOM$ holds relative to some $A \in \mathcal S$; that is, speaking more formally, suppose that there is no $g \in \mathcal S$ which dominates~$\F^A_{\Tot}$ and that there is an $f \in \mathcal S$ which is eventually different from every function in~$\F^A_{\Tot}$. Notice that $\F^{A }_{\Tot} \subseteq F^{A\oplus f}_{\Tot}$, which implies that $\neg \DOM$ also holds relative to $A \oplus f$. 
Suppose for a contradiction that~${f(e) = \Phi^A_e(e)}$ for $\mathcal M$-infinitely many~$e$, and define an $A \oplus f$-computable function~ $s$ by, for every~$n$, letting
$s(n)$ be the least number~$t$ such that for at least $n$ many $e<t$ we have
$\Phi^A_{e,t}(e) \da = f(e)$.	 Because $\neg \DOM$~holds relative to $A \oplus f$, there is a function $g \in \mathcal S$ such that $s(n) \leq g(n)$ $\mathcal M$-infinitely often. But then the function $h \in \mathcal S$ defined for all $e$ via
\[h(e) = \begin{cases}
		   \Phi_{e,g(e)} & \tn{if } \Phi_{e,g(e)}\da, \\
		   0&\tn{otherwise,} \\
	   \end{cases}\]
is total and $\mathcal M$-infinitely equals~$f$. This contradiction concludes the proof.
\end{proof}
 To conclude this section, we will consider the role of~$\MED$ in models of~$\neg \ISt$. In the presence of $\BSt$, we obtain the below result; however, as was the case for Proposition~\ref{prop:mad-no-ist}, this could be seen partly as a consequence of our definition of \emph{finite} for weakly represented families. What the situation looks like in the absence of $\BSt$ is left as an open question. We begin with a lemma. It is essentially about purely first-order models, but we phrase it in terms of topped models.
 \begin{lem}
  Suppose $\mathcal M \models \RCA + \BSt + \neg \ISt$ is topped by $A$. Let  $\Phi^A$ be universal in the sense of Lemma~\ref{lemma-universal}, suppose $c \in M$ is large enough that $\{e < c \colon \Phi_e^A$ is not $\mathcal M$-finite$\}$, and define a partial function $u(e,x)$ via
  \[u(e,x) = \min\{s\colon (\forall y < x)\Phi^A_{e,s}(y) \da \}.\]
  Then for each total $\Phi^A_e$ there is an $i < c$ such that $u(e,x) < u(i,c)$ $\mathcal M$-infinitely often.
 \end{lem}
 \begin{proof}
  Suppose for a contradiction there is a $\Phi_e^A$ which is total and such that
  \[(\forall i < c)(\exists d)(\forall x > d)[\tn{either~}u(i,x)\ua\tn{~or~}u(i,x) < u(e,x)].\]
  Using $\BSt$ to exchange the first two quantifiers, it follows that
  \[(\exists d)(\forall i < c)(\forall x > d)[\tn{either~}u(i,x)\ua\tn{~or~}u(i,x) < u(e,x)].\]
  But this means, with knowledge of $d$ and $e$, we see that the set $\{i < c \colon \Phi_i^A$ is total$\}$ is $\Pi^0_1$ and hence $\mathcal M$-finite, a contradiction.
\end{proof}

The next proposition shows that for topped models satisfying $\BSt + \neg \ISt$,
the axiom $\neg \AVOID$ coincides with $\MED$. Topped models satisfy that there
is an oracle $A$ in the second order part such that every function is
$A$-recursive. Then the $A$-recursive functions form a weakly represented family
and thus all functions are in this family, thus this family witnesses
$\neg \AVOID$. This motivates the following proposition which shows
that under $\RCA + \BSt + \neg \ISt$ topped models satisfy $\MED$.

\begin{prop}
  Every topped model of $\RCA + \BSt + \neg \ISt$ satisfies $\MED$.
\end{prop}
\begin{proof}
Fix $\mathcal M$, $A$, $\Phi$, $u$ as in the Lemma.
Assume (perturbing the definition if necessary) that $u$ is injective, so that we can list all its values in strictly ascending order:
\[u(e_0,x_0) <u(e_1,x_1) < u(e_2,x_2) < \cdots.\]

We define by stages a set $(g_i)_{i < c}$ of partial functions; at stage $s = u(e_i,x_i)$, we define $g_i(x_i)$. Begin by forming the set $E = \{ e < c \colon \Phi^A_{e,s} \da$ and $(\forall j \neq i) g_j(x) \neq \Phi^A_{e,s}(x)\}$. If $E$ is empty, set $g_i(x)$ to equal some number that has not been seen previously in the construction. Otherwise, select an $e$ which minimises ${e + |\{ y < x \colon \Phi_{e,s}(y) = g_i(y)\}|}$. This completes the constrution.

We make four claims about $(g_i)_{i<c}$. The first is that if $i\neq j$, then $g_i(x) \neq g_j(x)$ for all $x$ which are both defined; this is direct from the construction. The second is that those $g_i$ which are total are eventually different; this claim follows {\em a fortiori} from the first. The third is that $g_i$ is total iff the domain of $g_i$ is unbounded in $M$ iff $\Phi_e^A$ is total; this is direct from the construction and the definition of $u(i,x)$. The fourth claim is that these $g_i$ form a maximal such family. For a contradiction, suppose the fourth claim is false, which since the model is topped means that there is some $\Phi^A_e$ which is total and eventually different from each total $g_i$. That is, formally,
\[(\forall i < c)(\exists d)(\forall x > d)[\tn{either~}g_i(x)\ua\tn{~or~}g_i(x) \neq \Phi_e^A(x)].\]

Using $\BSt$ to exchange the leading quantifiers, this becomes
\[(\exists d)(\forall i < c)(\forall x > d)[\tn{either~}g_i(x)\ua\tn{~or~}g_i(x) \neq \Phi_e^A(x)].\]

Keeping $e$ fixed, let $i$ be as provided by the Lemma. Consider ${n(x) = e + |\{y < x \colon \Phi^A_{e,s}(y)=g_i(y)\}|}$, the function used during the construction of~$g_i$. Clearly $n$ is nondecreasing; and the equation above tells us is that $n(x)$~reaches some $\mathcal M$-finite limit $N$, after some stage $s = \langle i,x\rangle$ of the construction.
  Then at each stage~${t = \langle i,x'\rangle}$ that is greater than $s$, we know that $\Phi_e^A(x')$ was not chosen to equal $g_i(x')$.
  Since $i$ is given by the Lemma, we know that $u(e,x) < u(i,x)$ for $\mathcal M$-infinitely many $x$; hence at $\mathcal M$-infinitely many stages, there was another $e ' \in E$ with
  \[e' + |\{ y < x \colon \Phi_{e',s}(y) = g_i(y)\}| < N. \]
But this is possible only if $e' < N$, and thus no more than $N - e'$ times. This contradicts the pigeonhole principle, and proves the fourth claim. 
Therefore $(g_i)_{i<c}$ is a MED family, as required.
\end{proof}

\section{Hyperimmunity in Reverse Mathematics}\label{section-ndf}

The following principle is a reverse mathematics formalisation of the existences of hyperimmune degrees using weakly represented families.
\begin{statement} [Hyperimmunity principle, \HIM]
Given any weakly represented family~$\mathcal F$ of functions, there exists
a function $g$ such that for all $f\in\mathcal F$, there are unboundedly
many $x\in M$ such that $g(x)>f(x)$; in other words, there is no $f\in\mathcal F$ such that $g$ is almost everywhere less than~$f$.
\end{statement}
%This can be understood as the opposite of a variant of the notion of topped models, in the following sense: A model fails to satisfy $\HIM$ if and only if there
%is a weakly represented family $\mathcal F$ in the model such that, for all
%functions $g$ in the second-order part in the model, there is an $f \in F$
%with $g(x) \leq f(x)$ for all $x$. 

The following statement is immediate.
 \begin{prop}
  $\RCA \vdash \DOM \rightarrow \HIM$. \qed
 \end{prop}
The reverse direction does not hold, as this example shows.
\begin{exmp}
Consider the $\omega$-model whose second-order part $\mathcal S$ consists of all
second-order objects whose Turing degree is low for Martin-L\"of random.
These oracles form a Turing ideal not containing any high Turing degree,\footnote{See, for instance, Nies~\cite[Chapter~5]{Nies}.}
thus $\DOM$ does not hold in this $\omega$-model. Furthermore,
every set in $\mathcal S$ is
$\Delta_2$ and $\mathcal S$ is not topped. Thus for every $A \in \mathcal S$ there is a $B \in \mathcal S$ with $B \gneq_T A$ hyperimmune relative to $A$.
Thus $\HIM$ holds.
\end{exmp}

The following theorem confirms the separation of these two notions
with an alternative proof; it separates~$\DOM$ from~$\HIM$ by showing that,
whereas $\DOM + \BSt$ implies full arithmetical induction as observed in
Theorem~\ref{dom-induction}, $\HIM + \ISt$ does not even imply $\ISthr$.
 \begin{mythm}
  $\RCA + \ISt + \HIM$ is $\Pi^1_1$-conservative over $\RCA + \ISt$.
 \end{mythm}
 \begin{proof}
Fix a model $\mathcal M = (M,\mathcal S)$ of $\RCA + \ISt$ topped by a set $A_0 \in\mathcal S$. Applying the Low Basis Theorem to an $M$-infinite $A_0$-computable binary tree without $A_0$-computable paths we obtain a set $A_1 \in \Delta_2^{A_0} \setminus  \Delta_1^{A_0}$ that is low relative to ${A_0}$ and such that $A_0 \in\Delta_1^{A_1}$. 
  
Given any $A_0$-computable approximation $(A_{1,s})_{s \in M}$ to $A_1$, 
its {\em settling-time $f$}, given by
   \[f(x) = \min\{s\colon (\forall s'>s) [f(s')=f(s)]\}\]
for all~$x$, is not dominated by any total $\Delta_1^{A_0}$ function (as otherwise $A_1\in\Delta_1^{A_0}$ after all). Furthermore, since $A_1$~is low relative to ${A_0}$, all $\Sigma_2^{A_1}$ sets are $\Sigma_1^{A_0'}$ and hence regular by Lemma~\ref{lem:regular}. Therefore, $\mathcal M_1 = \mathcal M_0[A_1]$ is a model of~$\ISt$. 
   
Repeating the argument, we can obtain $M_n$ for all $n < \omega$, and finally we can let $M_\omega$ be the model obtained in the limit. Then $\mathcal M_\omega \models \RCA + \ISt + \HIM$, as desired.
\end{proof}

\section{Bi-immunity}\label{section-mind}

%Given a set $A \subseteq M$, let us write ${A}^{0}$
%for $A$ and ${A}^{1}$ for its complement $\bar A$.
%A $\mathcal M$-finite sequence of sets $A_0,\ldots,A_n$, represented as a set $A \in\mathcal S$ of pairs $\{\langle i,x \rangle: x \in A_i\}$, is said to be \emph{independent} if for any sequence $\sigma \in 2^n$, the set ${\bigcap}_{i < n}{{A}^{\sigma(i)}_{i}}$ is infinite. An \emph{independent family} is an $\F$ weakly represented by some $\Psi$ such that for every $\mathcal M$-finite sequence $\Psi_{f(0)},\cdots, \Psi_{f(n)}$ whose members are total characteristic functions, is independent,
%A \emph{maximal independent family} is an
%independent family that is not properly contained in another 
%independent family. [These definitions pose some problems when $\ISt$ does not hold, and only become fully satisfying when $\BSthr$ does hold, because of regularity properties. But we will not worry about that here.]

%\begin{statement}[Maximal Independent family principle \MIND]
%	There exists a weakly represented family of sets that is
%	maximal independent.
%\end{statement}

\begin{statement}[Bi-immunity principle, \BIM]
For every weakly represented family $\mathcal{F}$ of infinite sets there is
a set $B$ such that there is no $A\in \mathcal{F}$ with
$A\subseteq B$ or $A\subseteq \overline{B}$.
\end{statement}

We can rephrase this statement in terms of the universal sets of Lemma \ref{lemma-universal} as follows: $\BIM$~states that, for every $C $, there exists a $B $ for which there is no infinite $\Delta^C_1$~set $A =\Phi_e^C$ such that either $A \subseteq B$ or $A \subseteq \overline B$.

The following is a formalization of the well-known fact that every hyperimmune set  computes a bi-immune set. Note that, in combination with our previous results, this creates a connection between $\BIM$ on the one hand and $\DOM$, $\MAD$ and induction strength on the other.
\begin{prop}
 $\RCA + \HIM \vdash \BIM$.
\end{prop}
     
\begin{proof}
	Let $\mathcal M = (M, \mathcal S)$ be a model of $\HIM$, and fix any $A \in \mathcal S$. It is sufficient to construct a set~$B$ such that no $C=\Phi^A_e$ is contained in~$B$ or~$\overline B$. 
	
	Consider the family of use functions of the functionals $\Phi^A_{e}$ for every~$e$, and let $f$ witness $\HIM$ with respect to that family.	In order to decide which elements we want to put into~$B$, we proceed as follows for each~$y$ in ascending order: 
	Check if there exists an $e < y$ such that~$\Phi^A_{e,f(y)}(y) \da$ but such that for all $x < y$, 
	\[\text{either} \quad \Phi^A_{e, f(y)}(x) \ua \quad \text{or} \quad \Phi^A_{e,f(y)}(x) \da = B(x).\]
	If no such $e$'s exist, then let $y \not \in B$; and if they do exist, then pick the least such~$e$ and put~$y$ into~$B$ or~$\overline B$ in such way as to ensure $B(y) \neq \Phi^A_e(y)$.  By $\ISo$, every total $\Phi_e^A$ eventually gets treated in this way.
\end{proof}
The converse result does not hold. This can either be shown  directly by constructing a model of $\RCA$ that satisfies $\BIM$ but not $\HIM$, or by employing the following conservation result.
%, which gives the
%result in the same way as the $\Pi^1_1$-conservativeness of $\DOM$ over $\RCA$
%gave that $\DOM$ does not follow from $\HIM$.
\begin{mythm}
$\BIM$ is $\Pi^1_1$-conservative over $\RCA$.
\end{mythm}
The proof is an imitation of that used by Harrington \cite[IX.2]{SS} in proving that $\WKL$ is $\Pi^1_1$-conservative over $\RCA$.
\begin{proof}
Suppose that $\mathcal M = (M, \mathcal S)$ is countable model of $\RCA$ topped by $A \in \mathcal S$. If we can construct an~${A_1 \subseteq M}$ such that $M_1 = M[A_1] \models \ISo$ and such that $A_1$ is bi-immune relative to $A$ (in the sense that for no $e$ is $\Phi_e^A$ the characteristic function of an $\mathcal M$-infinite set which is either a subset of $A_1$, or a subset of~$\overline{A_1}$), then we can iterate the construction to obtain $\mathcal M_2 = \mathcal M_1[A_2]$, and so on, until we obtain in the limit a model $M_\omega$ of $\RCA + \BIM$. 
	
	\medskip
	
	We construct $A_1$ as the limit of a sequence of initial segments $(\sigma_i)_i$, as follows. To begin, let $(B_i)_{i < \omega}$~be an enumeration of all $\mathcal M$-infinite $A$-computable sets, and let $(c_i)_{i < \omega}$ be a cofinal sequence in~$M$, both ordered by the true natural numbers $\omega$.
\begin{itemize}
	\item At stage $0$, let $\sigma_0$ be the empty string.
 \item At stages of the form $s+1 =2i+1$ for some~$i$, let $\sigma_{s+1}$ be the shortest extension of~$\sigma_s$ such that $\sigma_{s+1}$ disagrees with (the characteristic function of) $B_i$.
	
	\item At stages of the form $s+1 = 2i+2$ for some~$i$, let $\sigma_{s+1}$ be the shortest extension of~$\sigma_s$ which maximises the size of $\{e < c_i \colon  \Phi_e(e) \da\}$.
\end{itemize}
The odd stages ensure bi-immunity, and the even stages guarantee that $\ISo$ holds.
\end{proof}

\end{document}